\documentclass[11pt]{amsart}
\usepackage{amsmath}
\usepackage{amssymb}
\usepackage{amsthm}
\usepackage{verbatim}
\usepackage{srcltx}
\usepackage{graphicx}
\usepackage[utf8]{inputenc}
\usepackage{tikz}
\usepackage[normalem]{ulem}

\usepackage{cancel}

\definecolor{darkbrown}{HTML}{431b00}

\newtheorem{thm}{Theorem}[section]

\newtheorem{lem}[thm]{Lemma}

\theoremstyle{definition}
\newtheorem{defn}[thm]{Definition}
\newtheorem{example}[thm]{Example}
\newtheorem{notation}[thm]{Notation}

\theoremstyle{remark}
\newtheorem{rem}[thm]{Remark}

\newcommand{\hv}{\operatorname{hv}}
\newcommand{\co}{\operatorname{co}}
\newcommand{\pc}{\text{prism}}
\newcommand{\rah}{\operatorname{rah}}

\def \co{\operatorname{co}}

\def\R{\mathbb{R}}

\def\2L{\Lambda_{\tilde{\gamma}}}
\def\1L{\Lambda_{\gamma}}

\renewcommand{\leq}{\leqslant}
\renewcommand{\geq}{\geqslant}

\begin{document}

\author{Pablo Angulo}
\address{ Department of Mathematics, Universidad Polit\'ecnica de Madrid}
\curraddr{}
\email{pablo.angulo@upm.es}

\author{Carlos García-Gutiérrez}
\address{ Department of Mathematics, Universidad Polit\'ecnica de Madrid}
\curraddr{}
\email{carlos.garciagutierrez@upm.es}

\thanks{The first author was supported by research grants MTM2014-57769-3-P, MTM2017-85934-C3-2-P, MTM2017-85934-C3-3-P, Severo Ochoa CEX2019-000904-S and PID2021-124195NB-C31 from the Ministerio de Ciencia e Innovaci\'on (MCINN), ERC 301179 and  ERC Advanced Grant 834728.}

\title{The $2+1$ convex hull of a finite set}

\begin{abstract}

  We study $\R^2\oplus\R$-\emph{separately convex} hulls of finite sets of points in $\R^3$, as in \cite{KirchheimMullerSverak2003}.
  This notion of convexity, which we call $2+1$ convexity, corresponds to rank-one convex convexity, or quasiconvexity, when $\R^3$ is identified with certain subsets of matrices.
  We introduce ``$2+1$ complexes'', which generalize $T_n$ constructions, define the ``$2+1$-complex convex hull of a set'', and prove that it is an inner approximation to the $2+1$ convex hull.

  We also consider outer approximations to $2+1$ convexity based in the locality theorem of rank convexity, by iteratively chopping off ``$D$-prisms''.
  For many finite sets, this procedure reaches a ``$2+1$ $K$-complex'' in a finite number of steps, and thus computes the $2+1$ convex hull.
  We show examples of finite sets for which this procedure does not reach the $2+1$ convex hull in a finite number of steps, but we show that there is always a sequence of outer approximations built with $D$-prisms that
  converges to a $2+1$ $K$-complex.
  We conclude that $K^{rc}$ is always a ``$2+1$ $K$-complex'', which has interesting consequences.

  \smallskip

\end{abstract}

\maketitle
\pagestyle{myheadings}
\markleft{P. ANGULO, %
C. GARCIA-GUTIERREZ}

\section{Introduction}

The purpose of this paper is to study in detail the $D$-convex hull of a finite set $K$ of  $\mathbb{R}^3$, for the particular case of \emph{separate convexity in $\mathbb{R}^2\oplus \mathbb{R}$}, which was introduced in \cite{KirchheimMullerSverak2003}. We shall call this $2+1$-convexity for short.
In contrast to the more common notion of separate convexity, the wave cone is not the cone over a finite set, which produces a number of interesting new phenomena.

For $2+1$-convexity any $5$ point set with non trivial $2+1$-convex hull has some rank one connection, or some $T_4$ (\cite[6.25]{KirchheimMullerSverak2003}).
However, an explicit $6$ point set $C_6$ was built (figure \ref{ex:kms6set_1}) with non trivial convex hull but without rank one connections or $T_4$'s. 

	\begin{figure}[h!]
		\includegraphics[width=6cm]{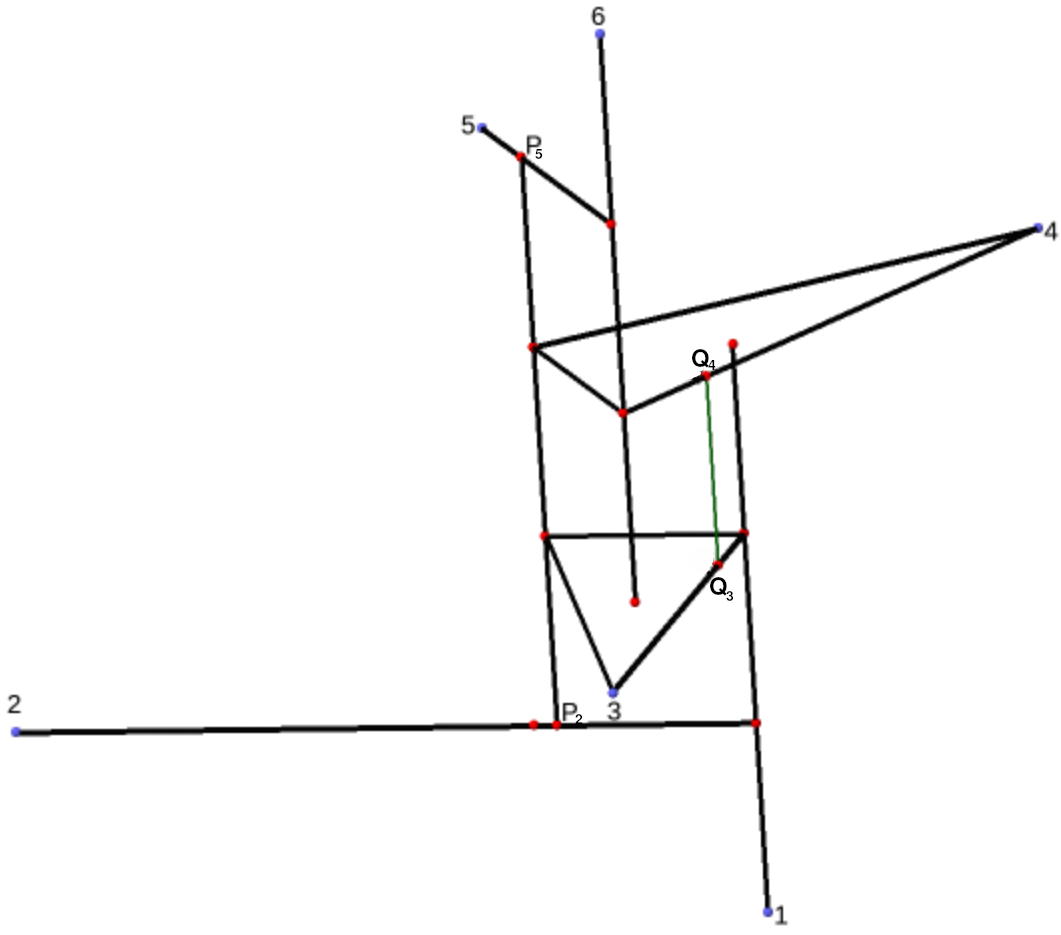}
		\caption{The $C_6$ point set described in the proof of \cite[6.26]{KirchheimMullerSverak2003}, together with its $2+1$-convex hull. $C_6$ consists of the six points labelled 1 to 6, such that the $z$-coordinate of point 1 is 1, etc.
			The auxiliar point $P\in\mathbb{R}^2$ is the intersection of segment that joins the projections of points 1 and 2, with the corresponding segment for points 5 and 6.
		}
		\label{ex:kms6set_1}
	\end{figure}
	
Another interesting feature of this set, very relevant from the computational point of view, is that its $2+1$-convex hull is described in terms of an auxiliary point $P\in\mathbb{R}^2$ which is not the projection of any of the 6 original points. This phenomena does not occur in separate convexity (\cite{MatousekPlechac1998,Matousek2001}).

$2+1$-convexity was introduced as a particular case of rank one convexity when $\mathbb{R}^3$ is identified with the following subset of $3\times 2$ matrices:
$$
\mathcal{M}=
\left\{
\left(\begin{array}{cc}
	x & 0\\
	y & 0\\
	0 & z
\end{array}\right):\, x,y,z\in\mathbb{R}\right\}.
$$

If we consider the set obtained by transposition of the matrices in $\mathcal{M}$, to get a subset of $2\times 3$ matrices, the wave cone induced in $\mathbb{R}^3$ by the rank one cone is exactly the same. While this transposition does not affect the laminate, rank one, or polyconvex hulls, it may affect the quasiconvex hull, since it is known that for ``diagonal'' $2\times 3$ matrices, functional quasiconvexity implies rank-one convexity (\cite{ HarrisKirchheimLin18},\cite{Muller99})
, while that statement is not known for ``diagonal'' $3\times 2$ matrices.

For a comprehensive and exhaustive presentation of the diverse notions of convexity, both for sets and funcions, we refer the reader to \cite{Dacorogna2006} and \cite{KirchheimMullerSverak2003}.
These notions of convexity are connected to different areas, as PDEs, calculus of variations, convex integration; or models for microstructures (\cite{Kirchheim2003,Dolzmann,KirchheimMullerSverak2003,Morrey1952,Tartar93}).
For these applications it is interesting to be able to efficiently compute the convex hull of a set for different notions of convexity.
In particular, in the context of convex integration, it is more natural to identify the set of matrices with an euclidean space $\mathbb{R}^k$, and study \emph{$D$-convex} hulls, for a more general \emph{wave cone} $D\subset \mathbb{R}^k$ (\cite{KirchheimMullerSverak2003, Kirchheim2003, SzekelyhidiTurbulence}).
	
For $D\subset \mathbb{R}^k$, it is known that the $D$-convex hull of a compact set $A$ is the zero set of the $D$-convex envelope of the function distance to $A$. This fact has been used to develop algorithms to estimate both $D$-convex envelopes of functions and $D$-convex hulls of subsets (see \cite{ArandaPedregal, Dolzmann, MatousekPlechac1998} and references therein).
Unfortunately, as shown in \cite{Matousek2001}, 
this approach for computing $D$-convex hulls of sets
requires computing $D$-convex envelopes of functions with an unrealistic precision (exponential in $|A|$).
	
From a geometrical perspective, in a few cases, some convex hulls can be calculated exactly and efficiently. For example, \cite{MatousekPlechac1998} presented an algorithm to calculate the separately convex hull of a finite set, which is related to rank one convexity and quasiconvexity, when $\mathbb{R}^k$ is identified with diagonal matrices.
Later, \cite{FranekMatousek} presented an algorithm to compute $D$-convex hulls exactly in the plane, where $D$ is the cone over a finite set of directions. His technique starts with an outer approximation which is iteratively reduced by a \emph{local bitting} algorithm.
For isotropic compact sets in the space of $2\times 2$ matrices, \cite{HeinzKruzik} developed a fast and exact algorithm to compute quasiconvex hulls.
Nevertheless, there is no exact algorithm for computing the rank one convex hull of a finite set of matrices in general.
	
More often than not, rank-one convex hulls and quasiconvex hulls are estimated through inner or outer approximations.
A complete determination of the rank-one, or the quasiconvex hull, is possible if an inner approximation is shown to agree with an outer approximation.
	
The first obvious inner approximation for the rank one convex hull $K^{rc}$ 
of a set $K$
is the lamination convex hull $K^{lc}$.
In contrast with standard convexity, the lamination convex hull is in general not sufficient to compute $K^{rc}$, but
other inner approximations have been used successfully.
The most well-known example is the $T_4$ configuration (see for example \cite{Muller99}), but a similar principle is used in more sophisticated examples (\cite{KirchheimMullerSverak2003, Kirchheim2003,Pompe2010,SebastyenSzekelyhidi,Harris2016}).
The 
structure of
those inner approximations is remarkably similar.
In section \ref{section:complexconvexhull}, we define \textbf{$2+1$-complexes}, and
the \textbf{$2+1$-complexes convex hull}, $K^{cc}$, in an attempt to make the inner approximation approach  more systematic.
For simplicity we define complexes only for $2+1$-convexity, but the extension to an arbitrary $D$ is natural. We are not aware of any example where $K^{cc}$ does not correspond to the $D$-convex hull $K^D$.

The polyconvex hull $K^{pc}$ is a well known outer approximation to the quasiconvex hull $K^{qc}$, which is an outer approximation to the rank one convex hull $K^{rc}$.
In some useful, but scarce, examples, $K^{pc}$ agrees with $K^{rc}$, but in general it is strictly larger.

The quasiconvex hull $K^{qc}$ is known to agree with $K^{rc}$ for some subsets of matrices \cite{Muller99}, although Šverák constructed sets whose rank-one convex hull is strictly smaller than their quasiconvex hull \cite{Sverak92}.

There are very few functions that are known to be quasiconvex but not polyconvex (\cite{Sverak92,Faraco-Zhong}), but they can be used to find outer approximations that are finer than the polyconvex hull.
In section \ref{section:polyconvex} we define the \textbf{polyconvex${++}$} convex hull $K^{pc++}$ of $K\subset \mathbb{R}^3$, which is an outer approximation to $K^{rc}$ finer than $K^{pc}$, which uses appropriate versions of  Šverák functions.
We show a specific $S\subset \mathbb{R}^3$ for which $S^{pc++}$ is a strict outer approximation to $S^{rc}$.
It was previously known that for separate convexity in $\mathbb{R}^3$, the analogue of $K^{pc++}$ is strictly larger than $K^{rc}$ (\cite{MatousekPlechac1998, Matousek2001}).

In section \ref{section:kirchheimballs} we define the $2+1$-prisms convex hull $K^{\pc}$, which consists of systematic application of theorem 4.7 in \cite{Kirchheim2003}.
It is very easy to prove that $K^{rc}\subset K^{\pc}$, but we will need to take limits of sequences of sets in order to prove our main theorem:

\begin{thm}\label{thm: maintheorem1}
	The $2+1$-convex hull $K^{rc}$ of any finite set $K\subset\R^3$ agrees with its $2+1$-complex convex hull and with its $2+1$-prism convex hull:
	$$K^{cc}=K^{rc}=K^{\pc}$$
	
	In particular, $K^{rc}$ is a $2+1$-complex.
\end{thm}

In section \ref{section: limit of outer approximations}, we show that $K^{\pc}$, and hence $K^{rc}$, can be approximated with a computer program.

We sum up the known inclusions and remark examples for which the inclusions do not hold.
\[
K^{lc}
\underbrace{\subsetneq}_{T_4}
K^{T_4}
\underbrace{\subsetneq}_{C_6}
K^{cc}
=
K^{rc}
=
K^{\pc}
=
K^{qc(2\times 3)}
\underbrace{\subsetneq}_{\text{section 3.3}}
K^{pc^{++}}
\underbrace{\subsetneq}_{T_4}
K^{pc}
\]

It remains an open question whether the rank one convex hull and the quasiconvex hull agree or not for $2+1$ convexity, when $2+1$ convexity is considered as a subset of $3\times 2$ matrices.

Our approach is to start with a straightforward outer approximation $M^0$ to $K^{rc}$ and then iteratively apply a step called ($2+1$, $K$) reduction.
At every step, the outer approximation $M_n$ is a $2+1$-complex (see def. \ref{def: 2+1 complejo}).
If all the \emph{extremal points} of some $M_n$ belong to $K$, lemma \ref{lem: A contenido 2+1 ch Extr A} proves that $M_n\subset K^{rc}$, and so they are equal.
The computations are exact, to the point that our implementation uses exact rational arithmetic.
However, if the points are in a ``generic position'', it is faster and qualitatively correct to use floating point arithmetic.
Often, this method computes $K^{\pc}$ in finite time, but we provide an example that shows that this is not the case in general, and $K^{\pc}$ can only be found in the limit.
However, even if taking a limit is necessary, the limit is always a finite $2+1$ complex whose extremal points belong to $K$.

As a by-product, the method proves the existence of a ``scaffolding'' of order $2$ for any finite set, which reduces the computation of the rank one convex hull to the computation of the lamination convex hull:
\begin{defn}\label{defn: scaffolding}
	A \textbf{scaffolding} of order $r$ of a finite set $K\subset \R^{m\times n}$ is a finite set $\tilde{K}\subset \R^{m\times n}$  that contains $K$ and such that:
	\begin{itemize}
		\item $\tilde{K}^{rc} = K^{rc} $
		\item $\tilde{K}^{lc,r} = \tilde{K}^{lc} = \tilde{K}^{rc}$
	\end{itemize}
\end{defn}
where $K^{lc,0}=K$, and, for any non-negative integer $r$, $K^{lc,r+1}$ is the union of all the rank one segments with endpoints in $K^{lc,r}$.

It is important that we impose a finite order for the $lc$-convex hull, since in general, the lamination convex hull can not be computed by adding rank one connections in a finite number of steps (\cite{Harris2016}).
The existence of a scaffolding of finite order is a direct consequence of the fact that $K^{rc}$ is a finite $2+1$-complex.
For $2+1$-convexity, a scaffolding can be found of order at most $2$.

\begin{thm}\label{thm: scaffolding}
	For any finite set $K$, there is a scaffolding $\tilde{K}$ of order $2$ for $K$, for $2+1$-convexity.
\end{thm}

The article is structured as follows: in section \ref{section:complexconvexhull} we define the $2+1$-complexes convex hull.
In section \ref{section:polyconvex} we use \emph{shovels} to define the
polyconvex$++$ hull $K^{pc++}$, and exhibit an example where
$K^{rc}\varsubsetneq K^{pc++}$. Section \ref{section:kirchheimballs} uses
Kirchheim's result \cite[Thm. 4.7]{Kirchheim2003} to prove that we can remove
from an outer approximation to the $2+1$-convex hull those extremal points that
do not belong ot $K$ to get a smaller outer approximation, using $D$-prisms.
Section \ref{section: limit of outer approximations} defines $K^{\pc}$ as a limit of outer approximations, and that the limit is always a $2+1$ $K$ complex, which concludes the proof of the main theorem.
Finally, in section \ref{section:conclusions} we discuss about some of the choices made, and possible generalizations.

\subsection{Acknowledgements}
We thank Daniel Faraco for introducing us into this fascinating subject and for many conversations.
We thank Jarmo Jäskeläinen, L\'aszlo Székelyhidi Jr, Bernd Kirchheim, Jan Kristensen and André Guerra for interesting conversations on this subject.

The software \texttt{sagemath} \cite{sage} was very useful for experimentation. It allowed us to write an exact implementation of ($2+1$, $K$) reduction, and perform many experiments.

\section{$2+1$ complexes and the $2+1$ complex convex hull}
\label{section:complexconvexhull}
The \textbf{lamination convex hull} of a subset of $\R^3$ is defined recursively:
\begin{defn}\label{defn: lamination convex hull}
  Let $A\subset\R^3$,
  \begin{itemize}
    \item
    $A^{lc,0}=A$
    \item
$
A^{lc,i+1}=\left\{
\lambda X+(1-\lambda) Y:
\begin{array}{l}
  X,Y\in A^{lc,i}\\
  \operatorname{rank}(X-Y)=1\\
  \lambda\in[0,1]
\end{array}
\right\}
$
  \item
  $A^{lc}=\bigcup_{i\geq 0}A^{lc,i+1}$
\end{itemize}
\end{defn}

The lamination convex hull is always contained in the rank one convex hull, but there are many examples in the literature of sets whose lamination convex hull is strictly smaller than the rank one convex hull.
The oldest and most popular example is the $T_4$ configuration (see Section 3.2 of \cite{KirchheimMullerSverak2003} for a definition). Apparently, this example was discovered by several authors at the same time (all references can be found in \cite{KirchheimMullerSverak2003}). For separate convexity, it is known that $T_4$-configurations and rank one connections fill the whole $K^{rc}$ (\cite{MatousekPlechac1998}), and for $2\times 2$ matrices, it is a very interesting open question (see \cite{Szekelyhidi05, FaracoSzeke08} for evidence in favour of the conjecture and some motivation).

This suggests the definition of a convex hull that contains $K^{lc}$ and is also contained in $K^{rc}$:
\begin{defn}\label{defn: T4 convex hull}
  Let $A\subset\R^3$,
  \begin{itemize}
    \item
    $A^{T_4,0}=A$
    \item
$
A^{T_4,i+1}=
\bigcup\left\{
F^{rc}:F\subset A^{T_4,i},
\begin{array}{l}
  F \text{ is a } T_4 \text{ configuration} \\
  \quad \text{  or a rank one connection}
\end{array}
\right\}
$
  \item
  $A^{T_4}=\bigcup_{i\geq 0}A^{T_4,i}$
\end{itemize}

\end{defn}

Nevertheless, the embedding of $T_n$-configurations sometimes fails to reach the rank one convex hull. The $C_6$ set (see figure \ref{ex:kms6set_1}), for $2+1$-convexity, has a nontrivial rank one convex hull and contains no $T_n$-configurations (\cite[6.26]{KirchheimMullerSverak2003}).

We will build an inner approximation to the $2+1$-rank one convex hull of a finite set of points that has not been shown to be different to the $2+1$-rank one convex hull.
We will do so by using $2+1$-complexes, and the inner approximation will be called \emph{$2+1$-complex convex hull}, $K^{cc}$ (see Defn. \ref{defn: 2+1 complexes convex hull}).
This approach comes naturally after reading \cite{SebastyenSzekelyhidi}, \cite{Pompe2010} or \cite{Harris2016}. To do so we shall need a couple of definitions and lemmas:

\begin{defn}
\label{def: 2+1 complejo}
    A \textbf{$2+1$-complex} $M$ is a closed and bounded subset of $\R^3$ that can be expressed as the union of a \textbf{finite list of elements} $L$, where each element is of the following kind:
\begin{description}
  \item[Point] a point of $\R^3$
  \item[Horizontal segment] A relatively open segment perpendicular to the $z$ axis whose two boundary points belong to $L$.
  \item[Vertical segment] A relatively open segment parallel to the $z$ axis whose two boundary points belong to $L$.
  \item[Horizontal triangle] A relatively open triangle contained in a plane perpendicular to the $z$ axis whose three boundary segments belong to $L$.
  \item[Vertical rectangle] A relatively open rectangle contained in a plane that contains the $z$ axis, whose four boundary segments belong to $L$. Two of those boundary segments are horizontal and two are vertical.
  \item[Open subset] An open subset of $\R^3$ whose boundary is a union of elements of the other kinds, and all of those elements belong to $L$.
\end{description}
\end{defn}

\begin{defn}\label{defn: extremal point}
    A point $p$ in a subset $M$ of $\mathbb{R}^3$ is \textbf{extremal} iff there is neither a vertical nor a horizontal relatively open segment contained in $M$ that contains $p$.
\end{defn}

This definition is consistent with definition 3.12 in \cite{Kirchheim2003}, for the particular case of $2+1$-complexes.

\begin{lem}
    The set $Extr(K)$ of extremal points of a $2+1$-complex $K$ is finite.
\end{lem}

\begin{proof}
    In a decomposition of $K$ into elements, only points can be extremal points, since the other elements are foliated by either horizontal or vertical lines, and by hypothesis the total number of elements is finite.
\end{proof}

The following lemma is a generalization of proposition 6.26 in \cite{KirchheimMullerSverak2003}.

\begin{lem}
\label{lem: A contenido 2+1 ch Extr A}
Let $M$ be a $2+1$-complex. Then
\[
M \subset [Extr(M)]^{rc}
\]
\end{lem}

\begin{proof}
Let $f:\R^3\rightarrow\R$ be a $2+1$-convex function that vanishes on $Extr(M)$.

We must prove that it vanishes on all of $M$.

Assume the contrary, and let $C>0$ be the maximum of $f$ on $M$ (which is compact), and let $M^*\subset M$ be the subset of $M$ where this maximum is attained: $M^*=M\cap f^{-1}(C)$.
There is at least one point $x^*\in M^*$ that is not contained in the interior of neither a horizontal nor a vertical segment contained in $M^*$.
However, since $f(x^*)=C$, the point $x^*$ is not in $Extr(M)$, and there is a segment $S$ contained in $M$ that is either horizontal or vertical and such that $x^*$ lies in its interior.
Thus $S$ cannot be contained in $M^*$.

However, since $f|S$ is convex and bounded by $C$, the only way it can attain the value $C$ at $x^*$ is if it constant on $S$.
This implies that $S$ is contained in $M^*$, which is a contradiction.
\end{proof}

\begin{defn}\label{defn: 2+1 K complex}
    For a finite set $K$, a \textbf{$2+1$ $K$-complex} is a $2+1$ complex whose extremal points are contained in $K$.
\end{defn}

\begin{defn}\label{defn: 2+1 complexes convex hull}
    The \textbf{2+1 complexes convex hull} of a finite set $K$ is the union of all $2+1$ $K$-complexes.
\[
K^{cc} = \displaystyle\bigcup_{\substack{M\text{ is }2+1\text{-complex}\\Extr(M)\subset K}} M
\]
\end{defn}

\begin{lem}
\label{lem:KcccontenidoKrc}
Let $K$ be a finite subset of $\R^3$.
\[
K^{cc}\subset K^{rc}
\]
where $K^{rc} $ is the $2+1$-convex hull of $K$.%
\end{lem}

\begin{proof}
It follows trivially from lemma \ref{lem: A contenido 2+1 ch Extr A}.%
\end{proof}

\begin{lem}
    The extremal points of $K^{cc}$ are contained in $K$:
\[
    Extr(K^{cc}) \subset K
\]
\end{lem}
\begin{proof}
  Let $x$ be an extremal point of $K^{cc}$.
  If there is some $2+1$ $K$-complex $M$ that contains $x$, and such that $x$ is not extremal for $M$, then $x$ would belong to the interior of either a vertical or horizontal segment contained in $M$, and that segment would also be contained in $K^{cc}$, and hence $x$ would not be an extremal point of $K^{cc}$.

  We deduce that $x$ is an extremal point of any $2+1$ $K$-complex $M$ that contains $x$.
  Since $x\in K^{cc}$, there is at least one such $M$, and this implies that $x\in K$.

\end{proof}

\section{The polyconvex$++$ hull}
\label{section:polyconvex}

The determinant of the minors of the $2\times 3$ matrices are polyconvex functions, which in our coordinates are the functions $xz$ and $yz$.
Any linear combination of them produces a polyconvex function.
In analogy with \cite{Sverak92}, \cite{Faraco-Zhong} and \cite{MatousekPlechac1998}, we consider the following functions:

\begin{defn}\label{defn: vertical quadrant}
  A \textbf{shovel function} is a function
  $$f(x,y,z)=\max\{l(x,y),0\}\cdot\max\{s(z-z_0),0\},$$
  for some $z_0\in \R$, an affine functional $l:\R^2\rightarrow \R$ and $s \in \{-1,1\} $.

  A \textbf{shovel} $L$ is the set where a shovel function is positive:
  \[
  L=\{(x,y,z)\in\R^3:\: l(x,y)>0, s (z-z_0)>0  \}
  \]
\end{defn}
A shovel function is a $2+1$ convex function, hence its zero set is $2+1$ convex, and shovels can be used to \emph{carve} $2+1$-convex sets.

\begin{defn}
  Given a compact set $K$, its \textbf{polyconvex$++$ hull} $K^{pc++}$ is the complement of the union of all the shovels that do not intersect $K$.
\end{defn}
\begin{rem}
  Since the complement of shovels are $2+1$ convex, the polyconvex$++$ hull of any set contains its rank one convex hull:
  $$ K^{rc}\subset K^{pc++}$$
\end{rem}

Shovels are analogous to the quadrants used in section §4 of \cite{MatousekPlechac1998} to compute separately convex hulls in the plane.
Quadrants are zero sets of quadrant functions:
$$
q(x_1,\dots,x_n) = \max\{s_1(x_1-x^0_1),0\}\cdot\dots\cdot \max\{s_n(x_n-x^0_n),0\}
$$
where $x^0=(x^0_1,\dots,x^0_n)\in\R^n$, and $s_i\in\{-1,1\}$ for $i=1\dots n$.

At the end of section §4 of \cite{MatousekPlechac1998}, there is a counterexample showing that quadrants do not compute the separately convex hull of a finite set in $\R^3$.

There is a limitation to that simple example: the separately convex hull of the six points is disconnected, and thus the separately convex hull can be computed for the subsets of points that generate each connected component alone, and then combined together to yield the whole hull using the ``structure theorem'' of rank one convex hulls: if $K^{rc}\subset U_1\cup U_2$ for disjoint open sets $U_1$ and $U_2$, then $K^{rc} = (K \cap U_1)^{rc}\cup (K \cap U_2)^{rc}$ (\cite{Kirchheim2003}, \cite{Matousek2001}, \cite{Pedregal1993}).
The quadrants are indeed enough to compute the separately convex hull of each subset of points.

However, in \cite{MatousekPlechac1998} there is also the $20$-point example 3.6, where the separately convex hull is connected, and quadrants does not compute $K^{rc}$, even if combined with the structure theorem.

We also provide a simple example that proves that for $2+1$-convex hulls shovels do not compute the $2+1$ convex hull, even if combined with the structure theorem.

\begin{example}\label{example: spiral staircase}
    Consider the \emph{Spiral Staircase set} consisting of the following points:
\[
   S = \{(1,0,0),(0,0,0),(0,0,1),(0,1,1),(0,1,2),(1,1,2)\}
\]

The lamination convex hull $S^{lc}$ is the union of the five segments that join consecutive points in $S$.
The polyconvex$++$ hull $S^{pc++}$ contains $S^{lc}$.
However, it also includes the horizontal triangle delimited by the three points $(0,0,1)$, $(0,1,1)$ and $(1/2,1/2,1)$ (points $3$, $4$ and $Q$ in figure \ref{figure: snake set}).
Indeed, let $f(x,y,z)=\max\{l(x,y),0\}\cdot\max\{s(z-z_0),0\}$ be a shovel function which vanishes on $S$.
If $z_0<1$ and $s=1$, then $l$ must vanish on the projections onto $\R^2$ of points 3, 4, 5, and 6, and then it must vanish on the projection of $Q$, hence $f$ vanishes on $Q$.
Similarly, if $z_0>1$ and $s=-1$, $l$ must vanish on the projections of the points 1, 2, 3 and 4, hence $f$ vanishes on $Q$.
In the other situations, it is $\max\{s(z-z_0),0\}$ which vanishes on $z=1$, the $z$ coordinate of $Q$.
In any case, $f(Q)$ is zero.

\begin{figure}[h!]
        \includegraphics[width=6cm]{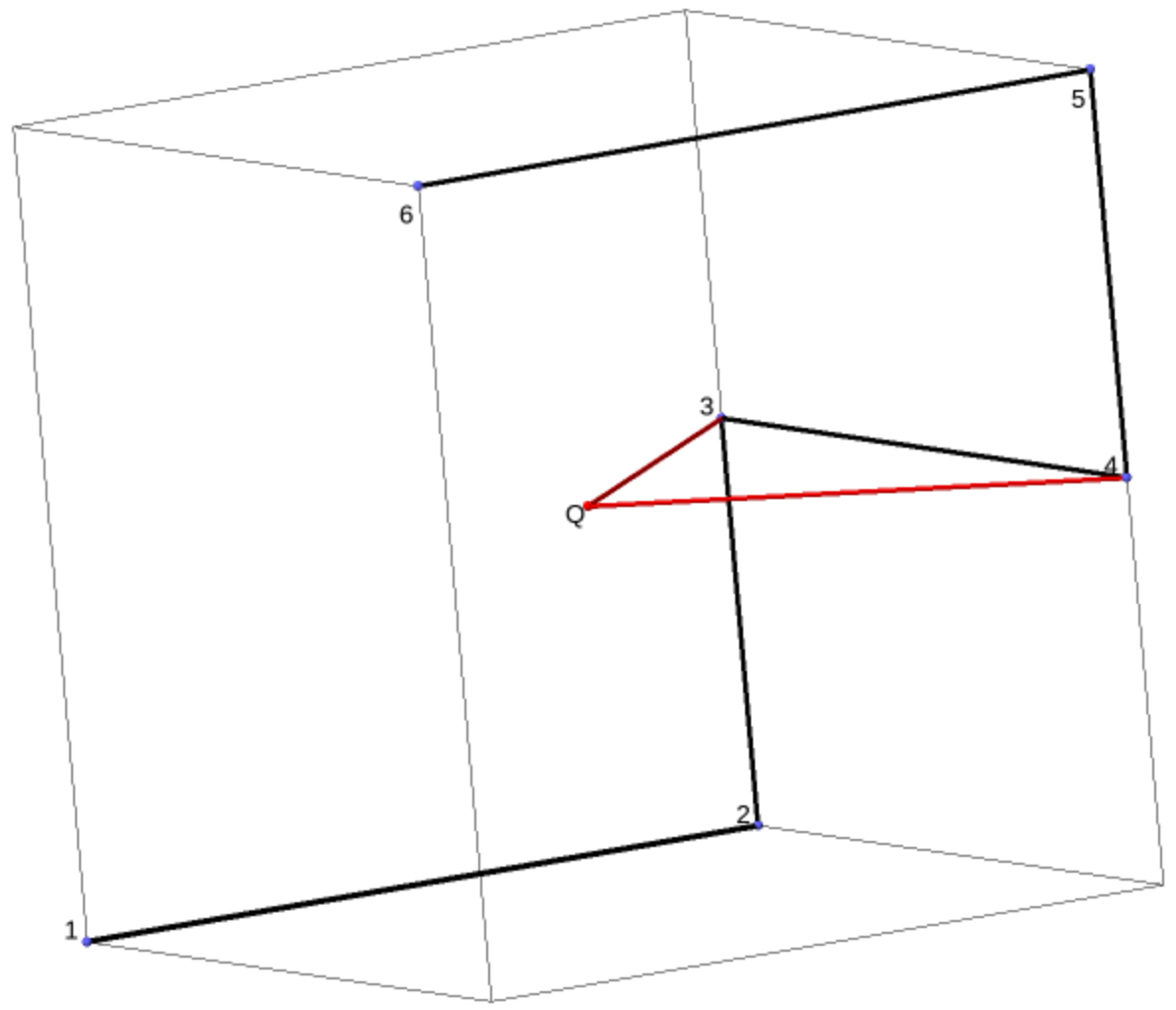}
        \caption{Spiral Staircase set}
        \label{figure: snake set}
\end{figure}
\end{example}

Thus $S^{pc++}$ contains an extremal point which is not in $S$.
We will see in section \ref{section:kirchheimballs} that $S^{rc}=S^{lc}$, hence $S^{rc}\subsetneq S^{pc++}$.

\section{$D$-prisms}
\label{section:kirchheimballs}

The result that makes our outer approximations work is the following well-known theorem

\begin{thm}[Thm. 4.7 from \cite{Kirchheim2003}]
\label{thm:kirchheimballs}
Let $B$ be bounded and $K$ compact in $\R^n$, then
\[
K^{rc}\cap B=\left[
(B\cap K)\cup (\partial B\cap K^{rc})\right]^{rc}\cap B
\]
\end{thm}
We will use this theorem to remove an extremal point from an outer approximation $M$ of the set $K^{rc}$ that does not belong to $K$, and reach a smaller outer approximation.
We will use the previous theorem only when $B$ is a $2+1$-prism.
\begin{defn}\label{defn: kirchheim prism}
  Let $K$ be a bounded subset of $\R^3$ and let $M\subset \R^3$ be a subset that contains $K^{rc}$.
  A \textbf{$2+1$-prism} for $(M,K)$ is the cartesian product $B=T\times I$ of an open (horizontal) triangle $T$ and an open (vertical) interval $I$ such that
\begin{itemize}
  \item $K\cap B = \emptyset$
  \item $M\cap \partial B$ is contained in the closure of the union of at most one of the three vertical rectangles and at most one of the two horizontal triangles that appear in the boundary of $B$.
\end{itemize}
\end{defn}

\begin{thm} \label{thm:kirchheimprisms}
 Let $M$ be an outer approximation for $K^{rc}$ and $B$ be a $2+1$-prism for $(M,K)$.

 Then $K^{rc}\subset M\setminus B$.

 Furthermore, if $M$ is $2+1$-convex, then $M\setminus B$ is also $2+1$-convex.
\end{thm}
\begin{proof}
  It follows from theorem \ref{thm:kirchheimballs} and $B\cap K=\emptyset $ that
\[
K^{rc}\cap B=\left[
(B\cap K)\cup (\partial B\cap K^{rc})\right]^{rc}\cap B =
 (\partial B\cap K^{rc})^{rc}\cap B
\]
Since $\partial B\cap K^{rc}\subset\partial B\cap M$ is contained in the union of a closed vertical rectangle and a closed horizontal triangle, which is a $2+1$-convex set, it follows that
\[
K^{rc}\cap B=
 (\partial B\cap K^{rc})^{rc}\cap B \subset
 (\partial B\cap M) \cap B\subset
 \partial B\cap B
\]

Since $B$ is open and $K^{rc}\subset M$, it follows that $K^{rc}\subset M\setminus B$.

If $M$ is $2+1$-convex, then $(M\setminus B)^{rc}\subset M^{rc}=M$.
Thus $M$ is an outer approximation for $(M\setminus B)^{rc}$, and $B$ is a $2+1$-prism for $(M,M\setminus B)$.
We have just proved that this implies that $(M\setminus B)^{rc}\subset M\setminus B$, hence $M\setminus B$ is $2+1$-convex.
\end{proof}

We will iteratively extract $2+1$-prisms from an initial $2+1$ complex $M_0$ that contains $K^{rc}$.
This produces a sequence of outer approximations $M_n$ that often estabilizes at a certain number $n^*$ of iterations, such that $M_{n^*}$ is a $2+1$ $K$-complex.
If we can achieve this, then lemma \ref{lem: A contenido 2+1 ch Extr A} proves that we have reached a set that is both an inner and an outer approximation to $K^{rc}$:
$$
M_{n^*}
\underbrace{\subset}_{2.6}
\left(
Extr(M_{n^*})
\right)^{rc}
\underbrace{\subset}_{M_{n^*}\text{is 2+1-$K$ complex}}
K^{rc}
\underbrace{\subset}_{\text{outer approximation}}
M_{n^*}
$$

In example \ref{example: spiral staircase}, we showed that shovels leave a stranded triangle (with vertices $3$, $4$ and $Q$ in figure \ref{figure: snake set}).
The stranded triangle can be removed, see figure \ref{figure: snake set triangle}, using a $D$-prism $P=T\times I$ with a triangle $T$ with vertices $q_1=(1/2+\varepsilon,1/2)$, $q_2=(0,0-\varepsilon)$ and $q_3=(0,1+\varepsilon)$ and an interval $I=(1-\varepsilon', 1+\varepsilon')$, for any $\varepsilon,\varepsilon'>0$ .
After removing that prism using theorem \ref{thm:kirchheimprisms}, we reach $S^{lc}$, so $S^{lc}=S^{rc}$.
If a point can be removed from a compact set with a shovel, it can also be removed by a $2+1$-prism, so $S^{rc}$ can be computed using outer approximations with $2+1$-prisms.

\begin{figure}[h!]
  \includegraphics[width=8cm]{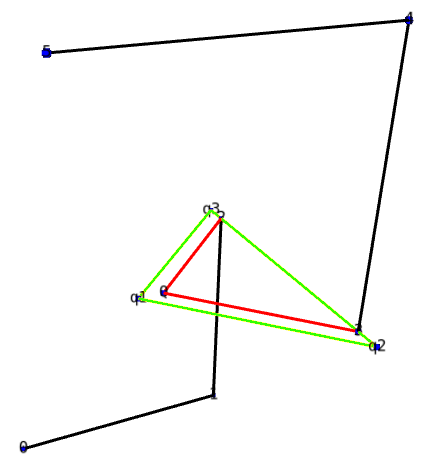}
  \caption{Spiral Staircase set with the triangle $T$ in green.}
  \label{figure: snake set triangle}
\end{figure}

However, for some finite sets $K$, no finite sequence of applications of theorem \ref{thm:kirchheimprisms} will compute the $2+1$ rank one convex hull.
We will study later examples \ref{example: infinitely many balls for one step of outer sequence} and \ref{example: unikorn6}, which have this property.

\begin{notation}
  Let $\pi:\R^3\rightarrow \R^2$ be the vertical projection onto the plane $\{z=0\}$, and
  let $\sigma:\R^3\rightarrow \R$ be the horizontal projection onto the line $\{x=y=0\}$.
\end{notation}

\begin{defn}
  Let $K\subset \mathbb{R}^3$ be a finite set.
  The crude outer approximation $C(K)$ to $K^{rc}$ is:
  $$C(K) = \co(\pi(K))\times\co(\sigma(K)),$$
  where $\co$ is the usual convex hull.
\end{defn}

\begin{defn}\label{defn: local 2+1 prism outer approximation}
  The local $2+1$-prism convex hull $K^{\pc}$ is the complement of the set of points that can be removed from $C(K)$ by a finite sequence of applications of Theorem \ref{thm:kirchheimprisms}.

$$
p\notin K^{\pc}\Leftrightarrow
\begin{array}{c}
\exists N\in\mathbb{N}, D_1,\dots,D_N\\
\text{ such that }\\
 D_j \text{ is a } 2+1 \text{-prism for} \left(C(K)\setminus\bigcup_{i=1}^{j-1} D_i,K\right) \text{ for } j=1,\dots,N\\
 \text{ and } p\in D_N
\end{array}
$$

\end{defn}
\begin{rem}
Any outer approximation that is bounded can be used instead of $C(K)$.
The usual convex hull $\co(K)$ works too, and we just used $C(K)$ because it is the start point of our algorithm.
\end{rem}

\section{A limit of outer approximations}\label{section: limit of outer approximations}
If we apply theorem \ref{thm:kirchheimprisms} iteratively, it is straightforward to build a sequence of outer approximations to $K^{rc}$.
Since this is a sequence of nested compact sets, it will have a limit.
However, if the outer approximations are not taken carefully, the limit set may not be $K^{rc}$.
It is easy to build examples of this, by using $2+1$-prisms that cut smaller and smaller portions of the outer approximation.

In this section, we prove that we can start with the crude outer approximation, and apply theorem \ref{thm:kirchheimprisms} in an specific way, in order to build a sequence of outer approximations $M_k$ that will converge to $K^{rc}$, and we will also be able to prove that the limit is a $2+1$ $K$-complex.
This will prove theorem \ref{thm: maintheorem1}.

In parallel, we will create a sequence of finite sets $A_k$ that are related to the outer approximations $M_k$ through the operator $\hv$, defined in \ref{defn: hv convex hulls}. We will genereate the sequence $A_k$ by iteratively applying the operator $\rah_K$, which is defined in \ref{defn: the kebab sequence}.
The advantage of using the sequence $A_k$ and the operator $\rah_K$ is that is can be coded into an efficient computer program.

\begin{defn}  \label{defn: h cuts}
Let $M\subset \R^3$, and let $h\in \R$.

The \textbf{$h$-cut} of $M$, $M^h$, is the projection of $M\cap \{z=h\}$ onto $\{z=0\}$.
$$
M^{h}=\pi(M\cap \{z=h\})
$$
\end{defn}

\begin{rem}
  A subset of $\R^3$ can be specified by giving its $h$-cuts for all $h$.
\end{rem}

\begin{defn}\label{defn: hv convex hulls}
  Let $ A\subset \R^3$ be a finite set, and let $H=\{h_1<\dots < h_n\}=\sigma(A)$ be the set of $z$-coordinates of all the points in $A$.

The \textbf{hv-convex hull}  of $A$, $\hv(A)$, is the set that results from taking the convex hull only in horizontal planes, and then taking convex hulls in the vertical direction in order to fill only the intermediate heights. It can be defined by its $h$-cuts:
\begin{itemize}
  \item
  For $h_j\in H$:
$$
\left(\hv(A)
\right)^{h_j} = \co(A^{h_j}).$$

  \item
  For $h\in (h_j,h_{j+1})$:
$$
\left(\hv(A)\right)^{h}
= \co(A^{h_j})\cap \co(A^{h_{j+1}}).$$
  \item For $h\not\in [h_1,h_n]$, $\left(\hv(A)\right)^h$ is empty.
\end{itemize}
\end{defn}

This is not the same as taking the laminate convex hull for $\{z=0\}$ and then the laminate convex hull for the $z$-axis, since for constructing $\hv(A)$ vertical segments are only used in intermediate heights.

\begin{rem}
  $\hv(A)
  $ is clearly a subset of $A^{lc,2}$, the lamination convex hull obtained with only two iterations of the ``add rank one connections'' operation that builds $A^{lc,i+1}$ from $A^{lc,i}$ in definition \ref{defn: lamination convex hull}.
\end{rem}

\begin{lem}
\label{lem:ahvis2+1complex}
  $\hv(A)$ is a $2+1$-complex.
\end{lem}
\begin{proof}
  $\hv(A)$ can be described as the union of one polygon at each height in $H$, which can be decomposed into triangles, and, for each $1\leq j<n$, one prism of the form $P=C\times (h_j,h_{j+1})$ for a finite 2D polygon $C$, and each prism $P$ can also be decomposed into a finite list of elements.
\end{proof}

\begin{defn}\label{defn:grid}
    A $(2+1)$ grid $G\subset\R^3$ (or a \textbf{grid} for short) is the product $F\times H$ of finite subsets $F\subset\R^2$ and $H\subset\R$.

    The \textbf{grid associated to a finite set $K\subset \R^n$} is the product $G=F\times H$ of the projection $F=\pi(K)$ of $K$ to the horizontal plane with the projection $H=\sigma(K)$ of $K$ to the vertical line.
\end{defn}

\begin{defn}\label{defn: the kebab sequence}
	{%
	Let $K\subset\R^3$ be a finite set , and let $H=\{h_1<\dots < h_n\}=\sigma(K)$ be the set of $z$-coordinates of all the points in $K$.

  Let $ A\subset \R^3$ be a finite set such that all $\sigma(A)\subset\sigma(K)$, and also $K\subset \hv(A)$.

  Then, for $h_j\in H$, the \textbf{($2+1$, $K$) reduction at height $h_j$} of the set $A$ is a finite set $A'=\rah_K(A,h_j)$ where only the level $h_j$ is altered:
}
  \begin{enumerate}
    \item The set
    $$\co \left(\left(\co (A^{h_{j+1}})\cap \co(A^{h_{j-1}})\right) \cup K^{h_j}\right)$$
    is a polygon in $\mathbb{R}^2$. Let $F$ be the set of its extremal points.
    Then $(A')^{h_j}=F$.

    {%
    If $j=1$ or $j=n$, then $A^{h_0}$ or $A^{h_{j_{n+1}}}$ are replaced by the empty set, and hence $(A')^{h_1}$ is the set of extremal points of $\co(K^{h_1})$.
    Similarly, $(A')^{h_n}$ is the set of extremal points of $\co(K^{h_n})$.
    It is not difficult to prove that $\co(K^{h_1})=(K^{rc})^{h_1}$ and $\co(K^{h_n})=(K^{rc})^{h_n}$, although this is not essential to the argument.}

    \item For any $h'\neq h_j$: $(A')^{h'}=A^{h'}$.
  \end{enumerate}

  The \textbf{discrete outer sequence} associated to a finite set $K$ is the sequence of finite sets $A_k$:
  \begin{enumerate}
    \item $A_0= G$, for the grid $G$ associated to $K$.
    \item $A_{k+1}$ is obtained from $A_k$ by performing one \textbf{($2+1$, $K$) reduction at height $h_j$} for each $h_j\in H$.
    All the reductions are performed simultaneously from $A_k$:
{%
	$$
	\left(
	A_{k+1}
	\right)^{h_j}=Extr
	\left(
	\co\left(\left(
	\co(
	A_k^{h_{j-1}})\cap
	\co(A_k^{h_{j+1}})\right)
	\cup K^{h_j}
	\right)
	\right)
	$$
}
  \end{enumerate}

  Associated to the outer sequence we also define the \textbf{outer sequence} associated to $K$, consisting of the sets $M_k=\hv(A_k)$.

  We say that the outer sequence is \emph{finite} if $A_k=A_{k+1}$ for some $k\in\mathbb{N}$.
  It follows then that $A_k=A_j$ for any $j\geq k$.
\end{defn}

Figure \ref{fig: outer sequence for the spiral staircase} shows the outer sequence for the set $S$ in example \ref{example: spiral staircase}.
The outer sequence for $S$ is finite.
\begin{figure}[!ht]
  \begin{tabular}{ccc}
    \includegraphics[width=3cm]{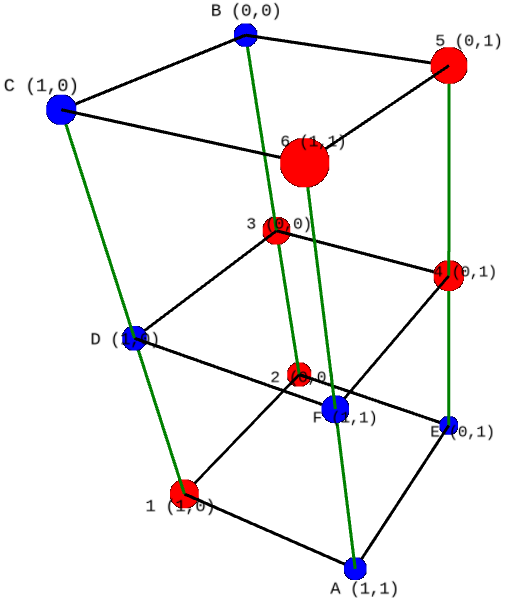} &
    \includegraphics[width=3cm]{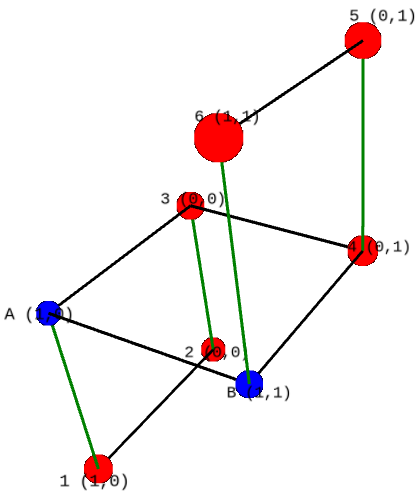} &
    \includegraphics[width=3cm]{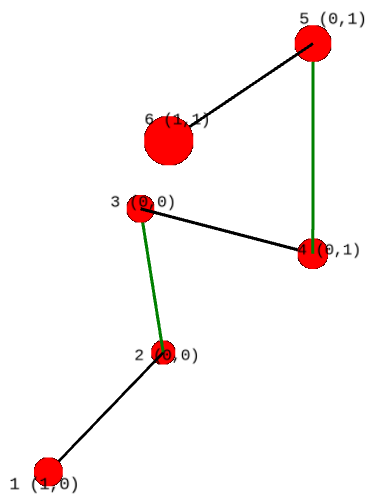}
  \end{tabular}
  \caption{$S$ consists of the six red points, while the other points in the grid associated to $S$ are shown in blue. The outer sequence converges in three steps, and no additional points are added to the grid.}
  \label{fig: outer sequence for the spiral staircase}
\end{figure}

By lemma \ref{lem:ahvis2+1complex} each step of the outer sequence goes from a $2+1$ complex $M_k$ to another $2+1$ complex $M_{k+1}$.
We will prove in lemma \ref{lem:krc is contained in the reduction at height} that $M_{k+1}$ is carved from $M_k$ by many applications of theorem \ref{thm:kirchheimprisms}, and it will follow that all the sets in the outer sequence are $2+1$-convex.

\begin{example}\label{example: infinitely many balls for one step of outer sequence}
  We will see in a specific example how $M_{k+1}$ is carved from $M_k$.
  This will require an infinite number of applications of theorem \ref{thm:kirchheimprisms}.
  For the sake of simplicity, let $k=1$, let $A$ be a set whose points take only
  three different values for their $z$-coordinate: $h_1<h_2<h_3$, and let $M=M_1=\hv(A)$.
  A $2+1$ reduction is performed at height $h_2$ (see figure \ref{fig: 2+1 reduction}):
\begin{itemize}
  \item $M^{h_1}$ is a convex polygon whose vertices are $i_1,i_2,i_3,p_-$.
  \item $M^{h_3}$ is a convex polygon whose vertices are $i_1,p^+,i_2,i_3$.
  \item $K^{h_2}$ consists of the two points $k_1,k_2$.
  \item $M^{h_1}\cap M^{h_3}$ is a convex polygon whose vertices are $i_1,i_2,i_3$.
  \item
  $M^{h_2}$ is the convex hull of $M^{h_1}\cup M^{h_3}\cup K^{h_{2}}$, and it is a convex polygon whose vertices are $i_2,i_3,k_1,p_-,p_+,k_2$.
  \item
  After a $2+1$ reduction at height $h_2$, a new $2+1$ complex $\widetilde{M}$ is reached, where
  $(\widetilde{M})^{h_2}$ is the convex hull of $\left(M^{h_1}\cap M^{h_3}\right) \cup K^{h_{j}}$, a convex polygon whose vertices are $k_1,k_2,i_2,i_3$.
\end{itemize}

\begin{figure}[!ht]
  \centering
  \begin{tabular}{cc}
    \includegraphics[width=0.45\textwidth]{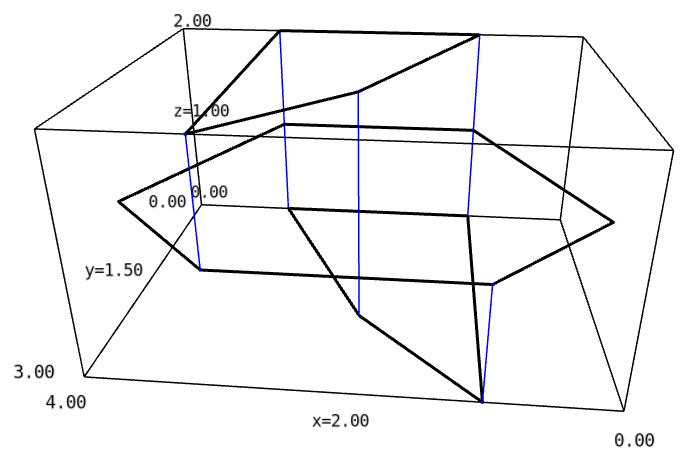} &
    \includegraphics[width=0.45\textwidth]{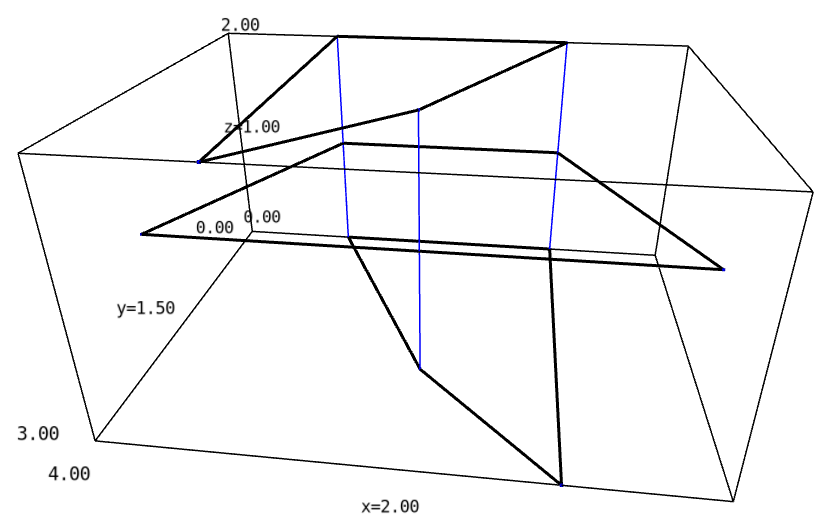}
  \end{tabular}
    \includegraphics[width=8cm]{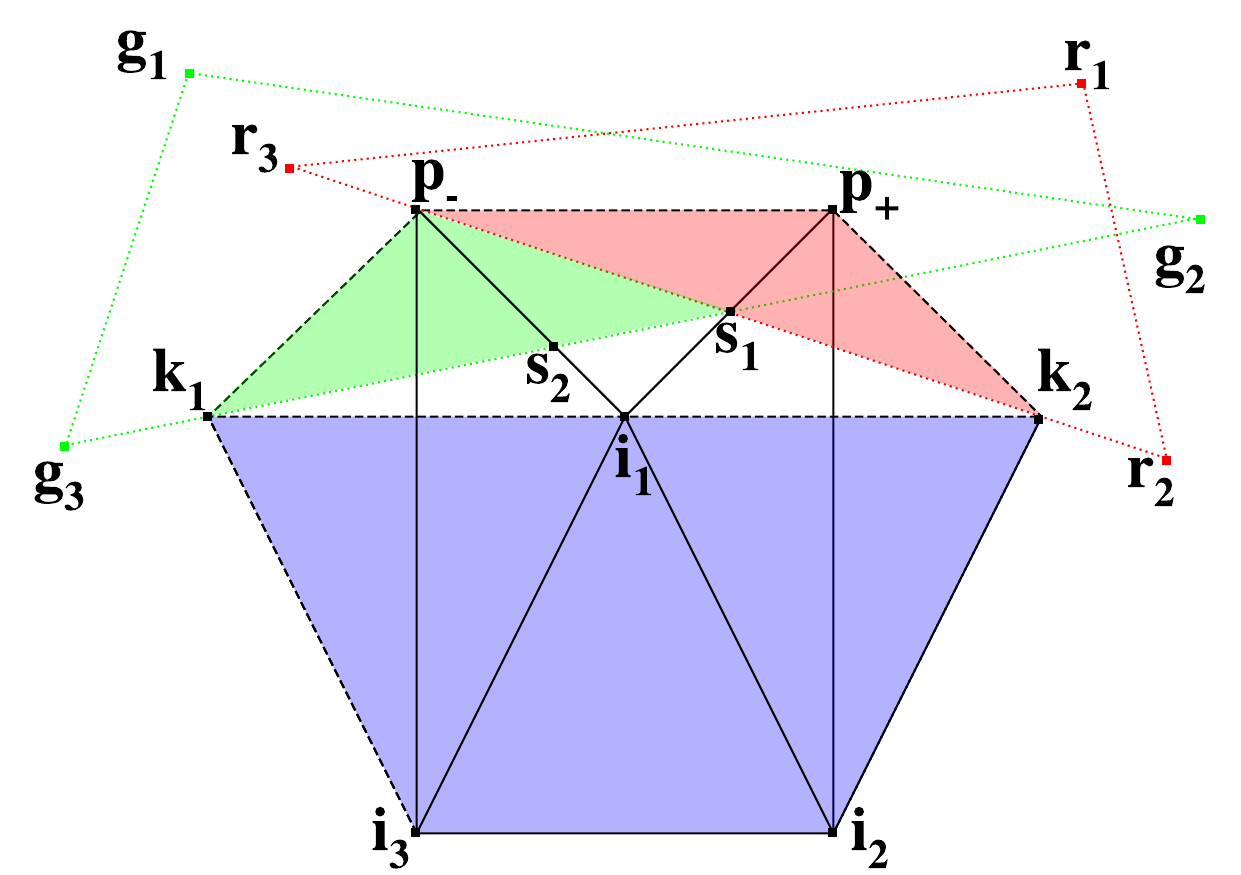}
  \caption{An example in which the $2+1$ reduction requires infinitely many applications of theorem \ref{thm:kirchheimprisms}.
  Top left: $M_k$; top right: $M_{k+1}$;
  bottom: projection to $\mathbb{R}^2$, as explained in the text.
  }
  \label{fig: 2+1 reduction}
\end{figure}

We first use prisms $P_\varepsilon=T\times(h_1+\varepsilon,h_3)$, where the open triangle $T$ has vertices $r_1,r_2,r_3$ and $\varepsilon$ is positive and smaller than $h_2-h_1$.
$P_\varepsilon$ is an open set that does not contain points of $K$, and
$\partial P_\varepsilon$ only intersects $M$ at its \emph{top} triangle
$T\times\{z=h_3\}$, and the vertical rectangle $\overline{r_2 r_3}\times
(h_1+\varepsilon,h_3)$, hence $P_\varepsilon$ is a $2+1$-prism for $(M,K)$, and by theorem \ref{thm:kirchheimprisms} it does not intersect $K^{rc}$.
Let $R$ be the triangle with vertices $k_2,p_+,p_-$ (the region shaded in red in figure \ref{fig: 2+1 reduction}).
Since $\varepsilon>0$ is arbitrary, $M'=M\setminus (R\times(h_1,h_3))$ also contains $K^{rc}$.

Only after theorem \ref{thm:kirchheimprisms} has been used with all the prisms $P_\varepsilon$, it is possible to use the prisms $P'_\varepsilon=T'\times(h_1,h_3-\varepsilon)$, where the open triangle $T'$ has vertices $g_1,g_2,g_3$ and $\varepsilon$ is positive and smaller than $h_3-h_2$.
The segment $\overline{g_2 g_3}$ contains the points $k_1$ and $s_1$, which is the intersection of $\overline{r_2 r_3}$ and $\overline{i_1 p_+}$.
$P'_\varepsilon$ does not contain points of $K$, and only intersects $M'$ at its \emph{bottom} triangle $T'\times\{z=h_{1}\}$, and the vertical rectangle $\overline{g_2 g_3}\times (h_1,h_3-\varepsilon)$, hence it does not intersect $K^{rc}$.
Let $G$ be the triangle with vertices $k_1,s_1,p_-$ the region shaded in green in figure \ref{fig: 2+1 reduction}.
Thus, for any $\varepsilon>0$, $P'_\varepsilon$ is a $2+1$-prism for $(M',K)$, and $M''=M'\setminus (G\times(h_1,h_3))$ also contains $K^{rc}$.

This iteration can continue, pivoting on points $s_1,s_2,s_3,\dots$ in the segments $\overline{i_1p_-}$ and $\overline{i_1p_+}$ that converge to $i_1$, and the remaining region of $\{z=h_2\}$ converges to the blue region, which is $\widetilde{M}^{h_2}=\co \left( \left(M^{h_1}\cap M^{h_3}\right) \cup K^{h_{2}}\right)$.
For intermediate heights $h\in(h_1,h_{2})$, the limit set is $\widetilde{M}^{h}=\widetilde{M}^{h_1}\cap \widetilde{M}^{h_2}$, and
for intermediate heights $h\in(h_{2},h_3)$, the limit set is $\widetilde{M}^{h}=\widetilde{M}^{h_{2}}\cap \widetilde{M}^{h_3}$.
The other heights $h\not\in(h_1,h_3)$ are not altered.
\end{example}

\begin{lem}\label{lem:krc is contained in the reduction at height}
Let $A$ be a finite set such that $\hv(A)$ is $2+1$-convex, $K$ be such that $K^{rc}\subset\hv(A)$, and let $h_j\in H=\sigma(A)$ be the height of some point of $A$.

Then $K^{rc}\subset \hv\left(
\rah_K(A,h_j)
\right)$, and $\hv\left(\rah_K(A,h_j)\right)$ is $2+1$-convex.
\end{lem}

\begin{proof}

The goal is to prove that $M_j=\hv\left(\rah_K(A,h_j)\right)$ is a subset of $M=\hv(A)$ obtained only by applications of theorem \ref{thm:kirchheimprisms} with $2+1$-prims that do not intersect $K$, which quickly proves the lemma.

In order to reach $M_j$ from $M$, theorem \ref{thm:kirchheimprisms} must be applied several times in a row.

We consider all prisms of the following types:
\begin{itemize}
  \item $P_{up}=T\times (h_{j-1}+\varepsilon,h_{j+1})$, where
    $0<\varepsilon<h_j-h_{j-1}$, and the vertices $q_1, q_2, q_3$ of $T$ satisfy
  \begin{itemize}
    \item $\overline{q_1q_3}\cap \pi(M)=\overline{q_1q_2}\cap \pi(M)=\emptyset$
    \item $T\cap M^{h_{j+1}}=\emptyset$
    \item $T\cap K^{h_{j}}=\emptyset$
  \end{itemize}

  \item $P_{down}=T\times (h_{j-1},h_{j+1}-\varepsilon)$, where
    $0<\varepsilon<h_{j+1}-h_j$, and the vertices $q_1, q_2, q_3$ of $T$ satisfy
  \begin{itemize}
    \item $\overline{q_1q_3}\cap \pi(M)=\overline{q_1q_2}\cap \pi(M)=\emptyset$
    \item $T\cap M^{h_{j-1}}=\emptyset$
    \item $T\cap K^{h_{j}}=\emptyset$
  \end{itemize}
\end{itemize}

We define $\widetilde{M}_j$ to be the complement of the set of points that can be removed from
$M$ by a finite sequence of applications of theorem \ref{thm:kirchheimprisms} with prisms of the above two types:

$$
p\notin \widetilde{M}_j\Leftrightarrow
\begin{array}{c}
\exists N\in\mathbb{N}, P_1,\dots,P_N \text{ such that }\\
 P_j \text{ is a prism of type } P_{up} \text{ or } P_{down}  \text{ above, for } j=1,\dots,N\\
 P_j \text{ is a } 2+1 \text{-prism for} \left(M\setminus\bigcup_{i=1}^{j-1} P_i,K\right) \text{ for } j=1,\dots,N\\
 \text{ and } p\in P_N
\end{array}
$$
By theorem \ref{thm:kirchheimprisms}, $\widetilde{M}_j$ is the intersection of a
set of $2+1$-convex sets, hence $\widetilde{M}_j$ is $2+1$-convex.
Since it contains $K$, it also contains $K^{rc}$.

We claim that  $\widetilde{M}_j=M_j$.
We divide the proof in three parts:

 \textbf{1. Proof that $\left( \widetilde{M}_j\right)^{h_j}\supset \left( M_j\right)^{h_j} = \co \left((M^{h_{j-1}}\cap M^{h_{j+1}})\cup K^{h_j}\right)$}.

Since we only use prisms that are contained in $\{h_{j-1}<z<h_{j+1}\}$, $ \widetilde{M}_j$ contains $M\cap\{z=h_{j-1}\}$ and $M\cap\{z=h_{j+1}\}$.
Since $\widetilde{M}_j$ is $2+1$-convex, it also contains the laminate convex hull of $M\cap\{z=h_{j-1}\}\cup M\cap\{z=h_{j+1}\}$, which contains $(M^{h_{j-1}}\cap M^{h_{j+1}})\times (h_{j-1},h_{j+1})$.
Since $\widetilde{M}_j$ contains $K^{h_j}$ and is $2+1$-convex, it follows that $\widetilde{M}_j$ contains the laminate convex hull of $\left((M^{h_{j-1}}\cap M^{h_{j+1}})\cup K^{h_j}\right)\times\{z=h_j\}$, which is $\co \left((M^{h_{j-1}}\cap M^{h_{j+1}})\cup K^{h_j}\right)\times\{z=h_j\}$.

\textbf{2. Proof that $\left( \widetilde{M}_j\right)^{h_j}\subset \left( M_j\right)^{h_j}$}

It follows from the definition \ref{defn: hv convex hulls} of $\hv$-convex hulls that $M^{h}=M^{h_{j-1}}\cap M^{h_{j}}\subset M^{h_{j-1}}$ for any height $h\in(h_{j-1},h_j)$.
Since we can use any prism $P_{up}$ with a smaller $\varepsilon$, so that
$0<\varepsilon<h_j-h_{j-1}$, every point $(x,y)$ that can be removed at height
$h_j$ can also be removed at height $h$, so that $(\widetilde{M}_j)^h\subset
M^{h_{j-1}}\cap(\widetilde{M}_j)^{h_{j}} $, for $h\in(h_{j-1},h_{j})$.
Similarly, $(\widetilde{M}_j)^h\subset M^{h_{j-1}}\cap(\widetilde{M}_j)^{h_{j}} $, for $h\in(h_{j},h_{j+1})$.

Let $p$ be an extremal point of the convex set $\left(\widetilde{M}_j\right)^{h_j}$, and assume it does not belong to $ \co \left((M^{h_{j-1}}\cap M^{h_{j+1}})\cup K^{h_j}\right)$.
We can assume, without loss of generality, that $p\not\in M^{h_{j-1}}$, since otherwise, $p\not\in M^{h_{j+1}}$ and the argument is similar.

We claim that $p$ does not belong to
$$R= \co \left(\left(M^{h_{j-1}}\cap \left(\widetilde{M}_j\right)^{h_j}\right)\cup K^{h_j}\right).$$
The point $p$ does not belong to neither $K^{h_j}$ nor $M^{h_{j-1}}$, so if $p\in R$, it belongs to the interior of a segment $\overline{q_1 q_2}$, where $q_1,q_2\in\left(M^{h_{j-1}}\cap \left(\widetilde{M}_j\right)^{h_j}\right)\cup K^{h_j}$.
But since $K^{h_j}\subset \left(\widetilde{M}_j\right)^{h_j}$ by part 1, it follows that
$$\co \left(\left(M^{h_{j-1}}\cap \left(\widetilde{M}_j\right)^{h_j}\right)\cup K^{h_j}\right)\subset \left(\widetilde{M}_j\right)^{h_j},$$
and then $p$ would not be an extremal point of $\left(\widetilde{M}_j\right)^{h_j}$.

Thus $p$ can be separated from $R$ by a line $L$.
In this line, we can find two points $q_2,q_3$, plus another point $q_1$ that lies to the same side of $L$ as $p$, so that $p$ belongs to the interior of the triangle with vertices $q_1,q_2,q_3$, and $\overline{q_1q_2}$ and $\overline{q_1q_3}$ do not intersect $\pi(M)$.
We build the prism $P_{up}=T\times(h_{j-1}+\varepsilon, h_{j+1})$.
Its bottom triangle does not intersect $\left( \widetilde{M}_j\right)^{h_{j-1}+\varepsilon}\subset M^{h_{j-1}}
\cap \left(\widetilde{M}_j\right)^{h_j}
\subset R$.

Thus $P_{up}$ intersects $\widetilde{M}_j$ only at its top triangle, and the vertical rectangle $\overline{q_2q_3}\times(h_{j-1}+\varepsilon, h_{j+1})$.
Furthermore, since $R$ contains $K^{h_j}$, $P_{up}$ does not contain points of $K$, hence $P_{up}$ is a  $(\widetilde{M}_j,K)$ prism.
This implies that $p$ can be removed from $\widetilde{M}_j$ by one of the prisms that define $\widetilde{M}_j$, which is a contradiction.
Since all extremal points of $\left(\widetilde{M}_j\right)^{h_j}$ belong to $\left( M_j\right)^{h_j}$, we deduce that $\left(\widetilde{M}_j\right)^{h_j}\subset \left( M_j\right)^{h_j}$, since both sets are convex.

\textbf{3. Proof that $\left( \widetilde{M}_j\right)^{h}= \left( M_j\right)^{h}$ for every $h$.}

We have established that $(\widetilde{M}_j)^{h_j}=(M_j)^{h_j}$.
By the nature of the prisms we have used, it is clear that $(\widetilde{M}_j)^{h}=(M_j)^{h}$ for any $h\leq h_{j-1}$ and for any $h\geq h_{j+1}$.
We proved in the previous part that $(\widetilde{M}_j)^h\subset (\widetilde{M}_j)^{h_{j}}\cap M^{h_{j-1}}$, and we have also proved that $(\widetilde{M}_j)^{h_j} = (M_j)^{h_j}$.
Since $M_j$ is $2+1$-convex, it must contain $\left((M_j)^{h_j}\cap (M_j)^{h_{j-1}}\right)\times(h_{j-1},h_j)$, and hence they agree.
The argument for $h\in (h_{j},h_{j+1})$ is similar, and we have established that $(M_j)^h=(\widetilde{M}_j)^{h}$ for any $h$.

This concludes the proof that $M_j$ is obtained from $M$ only by applications of theorem \ref{thm:kirchheimprisms}, from which it follows that $M_j$ is $2+1$-convex, and since no point of $K$ was removed it follows that $K^{rc}\subset M_j$.
\end{proof}

\begin{thm}\label{thm: the kebab sequence consists of outer approximations}
  Each $M_k$ in the outer sequence of a finite set $K$ is a $2+1$ convex set that contains $K^{rc}$.
\end{thm}
\begin{proof}
  The result is proved by induction on $k$.
  We assume that $M=M_k=\hv(A_k)$ is a $2+1$-convex set that contains $K^{rc}$, and we must prove that $M_{k+1}$ also is.

  Let $M_j=\hv(\rah_K(A_k,h_j))$ and $M'=\hv(A_{k+1})$.
  We will prove that $M' =\cap_{j=1}^{n} M_j$.

  Let us describe $(M_j)^h$ for any $h$:
  \begin{itemize}
    \item $(M_j)^{h_j}=\co \left((M^{h_{j-1}}\cap M^{h_{j+1}})\cup K^{h_j}\right)$.
    \item $(M_j)^{h}= \co \left((M^{h_{j-1}}\cap M^{h_{j+1}})\cup K^{h_j}\right)\cap M^{h_{j-1}}$ for $h\in(h_{j-1}, h_j)$.
    \item $(M_j)^{h}= \co \left((M^{h_{j-1}}\cap M^{h_{j+1}})\cup K^{h_j}\right)\cap M^{h_{j+1}}$ for $h\in(h_{j}, h_{j+1})$.
    \item $(M_j)^{h}= M^{h}$ for $h\not\in(h_{j-1}, h_{j+1})$.
  \end{itemize}
  From this description we derive the following description of $\cap_{j=1}^{n}M_j$ for any $h$:
  \begin{itemize}
    \item $\left(\cap_{j=1}^{n}M_j\right)^{h_j}=\co \left((M^{h_{j-1}}\cap M^{h_{j+1}})\cup K^{h_j}\right)=(M')^{h_{j}}$ for $h_j\in H$.
    \item $\left(\cap_{j=1}^{n}M_j\right)^{h}= (M')^{h_{j-1}}\cap M^{h_{j-1}}\cap (M')^{h_j}\cap M^{h_{j}}=(M')^{h_{j-1}}\cap (M')^{h_j}$ for $h\in(h_{j-1}, h_j)$.
  \end{itemize}

  It follows from definition \ref{defn: the kebab sequence} that $M'=\cap_{j=1}^{n}M_j$.

Now, using lemma \ref{lem:krc is contained in the reduction at height} we have
$$
K^{rc}\subset \cap_j\hv(\rah_K(A_k,h_j)) = \hv(A_{k+1})=M'
$$
and $M'$ is a $2+1$-convex set.

\end{proof}
 
\begin{lem}
  If the outer sequence is finite and stops at $M_{n^*}$, then $M_{n^*}=K^{rc}$.
\end{lem}
\begin{proof}
  Let $p$ be an extremal point of $M_{n^*}$ with height $h_j$.
  Then $\pi(p)\not\in (M_{n^*})^{h_{j-1}}\cap (M_{n^*})^{h_{j+1}}$, since otherwise $p$ would belong to the interior of a vertical segment contained in $M_{n^*}$.
  Hence, if $p$ does not belong to $K$, then either $p$ will not belong to $M_{n^*+1}$ or it will belong to the interior of an horizontal segment in  $M_{n^*+1} = M_{n^*}$.
  Both of these options are contradictions.

  It follows that $M_{n^*}$ is a $2+1$ $K$-complex, and lemma \ref{lem: A contenido 2+1 ch Extr A} and theorem \ref{thm: the kebab sequence consists of outer approximations} prove that we have reached a set that is both an inner and an outer approximation to $K^{rc}$:
  $$
  M_{n^*}
  \underbrace{\subset}_{2.6}
  \left(
  Extr(M_{n^*})
  \right)^{rc}
  \underbrace{\subset}_{M_{n^*}\text{is 2+1-$K$ complex}}
  K^{rc}
  \underbrace{\subset}_{\ref{thm: the kebab sequence consists of outer approximations}}
  M_{n^*} = K^{\pc}
  $$
\end{proof}

In our next example, however, the outer sequence is not finite.

\begin{example}\label{example: unikorn6}
Let
$$
  K=
  \left\{
\begin{array}{rrr}
  P_1=(2,0,0), &P_2=(2,2,0), &P_3=(0,1,1), \\
  P_4=(3,1,2), &P_5=(1,0,3), &P_6=(1,2,3)
\end{array}
\right\}.
$$
$K$ consists of the red points in figure \ref{fig: unikorn6}.
The red segments, plus the vertical square that they delimit, constitute a $2+1$ $K$ complex, which we name $L$, whose extremal points are exactly the points of $K$.
We can describe $L$ as $\hv(B)$ for a set $B$ with $8$ points.
In order to write that description we define two more points: $Q_1=(2,1)$ is the intersection of the segment $\overline{\pi(P_1)\pi(P_2)}$ with the segment $\overline{\pi(P_3)\pi(P_4)}$, and $Q_2=(1,1)$ is the intersection of the segment $\overline{\pi(P_5)\pi(P_6)}$ with the segment $\overline{\pi(P_3)\pi(P_4)}$.
\begin{itemize}
  \item $B^0$ consists of $\pi(P_1)$ and $\pi(P_2)$.
  \item $B^1$ consists of $\pi(P_3)$ and $Q_1$.
  \item $B^2$ consists of $\pi(P_4)$ and $Q_2$.
  \item $B^3$ consists of $\pi(P_5)$ and $\pi(P_6)$.
\end{itemize}

  \begin{figure}[!ht]
    \begin{tabular}{cc}
      \includegraphics[width=5cm]{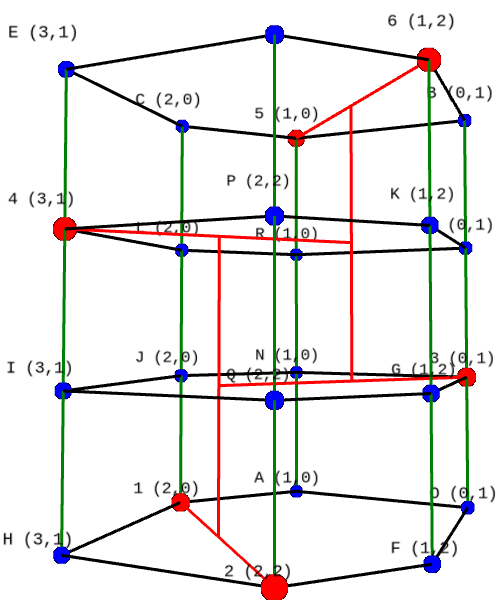} &
      \includegraphics[width=5cm]{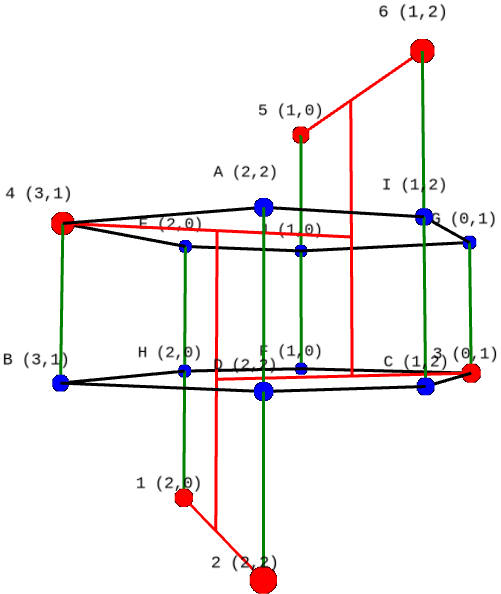} \\
      \includegraphics[width=5cm]{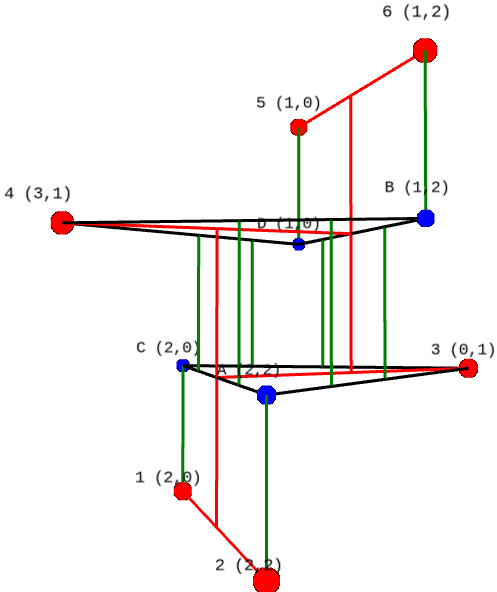} &
      \includegraphics[width=5cm]{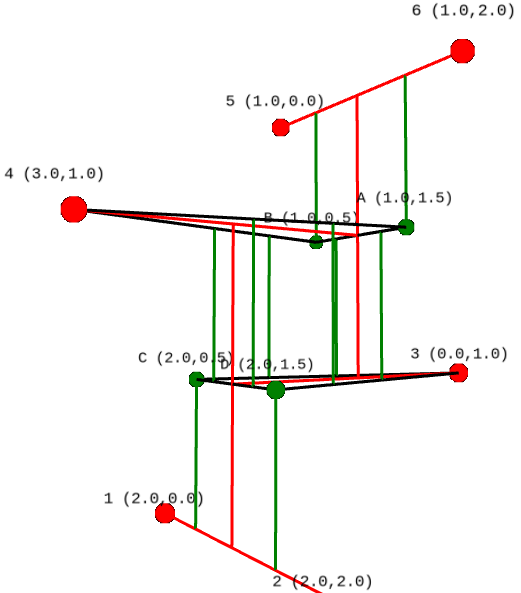}
    \end{tabular}
    \caption{Top left: initial crude approximation $M_0=\hv(G)$; top right: element $M_1$ of the outer sequence; bottom left: element $M_2$; bottom right: element $M_3$.
    $K$ consists of the six red points. $K^{rc}$ is shown as red lines. Blue points belong to the original grid. Green points are generated at step $3$, and are extremal in $M_3$, but they do not belong to $K$.}
    \label{fig: unikorn6}
  \end{figure}

By lemma \ref{lem: A contenido 2+1 ch Extr A}, $L$ is contained in $K^{rc}$.
After the first two iterations of the outer algorithm, $A_n$ retains the same configuration: the segment $\overline{P_5 P_6}$ at the top and the segment $\overline{P_1 P_2}$ at the bottom, and two triangles at the middle levels, but the vertices of those triangles change with $n$.

Let us look at the derivation of $M_3$ from $M_2$.
$M_2$ consists only of points of the original grid $G$, but four of them (labelled $A,B,C,D$ in figure \ref{fig: unikorn6}, bottom left) are extremal and do not belong to $K$.
The triangle at height $2$ has vertices $\pi(P_4)=(3,1), D=(1,0), B=(1,2)$.
The segment at height $0$ has endpoints $\pi(P_1)=(2,0)$ and $\pi(P_2)=(2,2)$.
The intersection of this segment and that triangle is the segment with endpoints $(2,0.5)$ and $(2,1.5)$.
Thus, the ($2+1$, $K$) reduction at height $1$ is $A'$ which differs from $A_2$ by the triangle at height $1$, which has vertices $(0,1)$, $(2,0.5)$ and $(2,1.5)$.
Analogously, the ($2+1$, $K$) reduction at height $2$ will alter only the triangle at height $2$, from one with vertices $(3,1), (1,0), (1,2)$ into one with vertices $(3,1), (1,1/2), (1,3/2)$.
Thus $A_3$ consists of the same segments $\overline{(2,0), (2,2)}$ at height $0$ and $\overline{(1,0), (1,2)}$ at height $2$, and two triangles in the middle levels.

  From step $2$ on, the top and bottom segments will remain unaltered, but the triangle in $M_{2+j}$ at height $1$ will have vertices $(0,1), (2,1-1/2^j), (2,1+1/2^j)$, and the triangle in $M_{2+j}$ at height $2$ will have vertices $(3,1), (1,1-1/2^j), (1,1+1/2^j)$, as the reader may verify by induction.

  Thus the outer sequence is not finite, but its limit is the inner approximation $L$.

  The passage from $M_{2+j}$ to $M_{2+j+1}$ consists of removing from $M_{2+j}$ all the $2+1$-prisms for $(M_{2+j},K)$.%
  In particular, a finite sequence of applications of theorem \ref{thm:kirchheimprisms} will not reach $K^{rc}$.
  Although clever application of theorem \ref{thm:kirchheimballs} directly can find $K^{rc}$ in finitely many steps, using sets $B$ more general than $2+1$-prims, it is not clear how to apply more general sets $B$ for arbitrary sets $K$.

\end{example}

\begin{thm}\label{thm: the limit is a finite complex}
  The outer sequence $\{M_n\}$ of a finite set $K$ converges to a $2+1$ $K$ complex.
\end{thm}
\begin{proof}
  Since $M_{n+1}\subset M_n$, and each $M_n$ is compact, the outer sequence is a sequence of nested compact sets, whose limit $M_\infty=\bigcap M_n$ is a compact set.
  Since each $M_n$ contains $K^{rc}$, $M_\infty$ also does.

  For any two heights $h,h'\in (h_j,h_{j+1})$, the sets $M_n^{h}$ and $M_n^{h'}$ are the same.
  Hence, for the limit set, it also holds that $M_\infty^{h} = M_\infty^{h'}$.

  For any $h\in (h_j,h_{j+1})$, we have $M_n^{h} = M_n^{h_j}\cap M_n^{h_{j+1}}$.
  It follows that $M_\infty^{h} = M_\infty^{h_j}\cap M_\infty^{h_{j+1}}$.

  \paragraph{\textbf{Claim:}}
   All extremal points of $M_\infty$ belong to $K$.

  Suppose on the contrary that $p=(x,y,h)\not\in K$ is an extremal point of $M_\infty$.

  Let $\bar{p}=\pi(p)=(x,y)$.

  Suppose first that $h$ lies in the interval $(h_j,h_{j+1})$.
  Since $M_\infty^{h} = M_\infty^{h'}$ for any other $h'\in (h_j,h_{j+1})$, the vertical segment $(x,y)\times (h_j,h_{j+1})$ contains $p$ and is contained in $M_\infty$.
  It follows that $p$ is not extremal.

  Suppose now that $h=h_j\in H$.
  Since $M_\infty^{h'}$ is the same set for any $h'\in (h_j,h_{j+1})$, we call this set $M_\infty^{h+}$. We define $M_\infty^{h-}$ similarly.
  Since $p$ is extremal, it follows that at least one of $M_\infty^{h+}$ and $M_\infty^{h-}$ do not contain $\bar{p}$.
  Without loss of generality, we assume $\bar{p}\not \in M_\infty^{h+}$.

  We claim that $\bar{p}\not\in \co \left(K^{h} \cup M_\infty^{h+} \right)$.
  Since $p\not\in K$ and $\bar{p}\not \in M_\infty^{h+}$, $\bar{p}\in \co \left(K^{h} \cup M_\infty^{h+} \right)$ implies that $\bar{p}$ lies in the interior of a segment with endpoints in $K^{h} \cup M_\infty^{h+}$.
  Since both $K^{h} $ and $M_\infty ^{h+}$ are subsets of $M_\infty^{h}$, this would also imply that $p$ is not extremal in $M_\infty$.

  Since $\bar{p}\not\in co \left(K^{h} \cup M_\infty^{h+} \right)$, we can separate $\bar{p}$ from $co \left(K^{h} \cup M_\infty^{h+} \right)$ with a line $L\subset \R^2$.

  Since $p$ is extremal, there is a point $q_1$ near $\bar{p}$ that does not belong to $ M_\infty^{h}$, and such that $q_1$ and $K^{h}$ lie at opposite sides of $L$.
  We also require that $q'$ is separated from $M_\infty^{h}$ by the line parallel to $L$ that goes through $\bar{p}$.
  Then we can find two points $q_2$ and $q_3$ in the line $L$ such that the segments $\overline{q_1 q_2}$ and $\overline{q_1 q_3}$ do not intersect $M_\infty^{h}$ (see figure \ref{fig: D prism that cuts p}).

  \begin{figure}[!ht]
      \includegraphics[width=8cm]{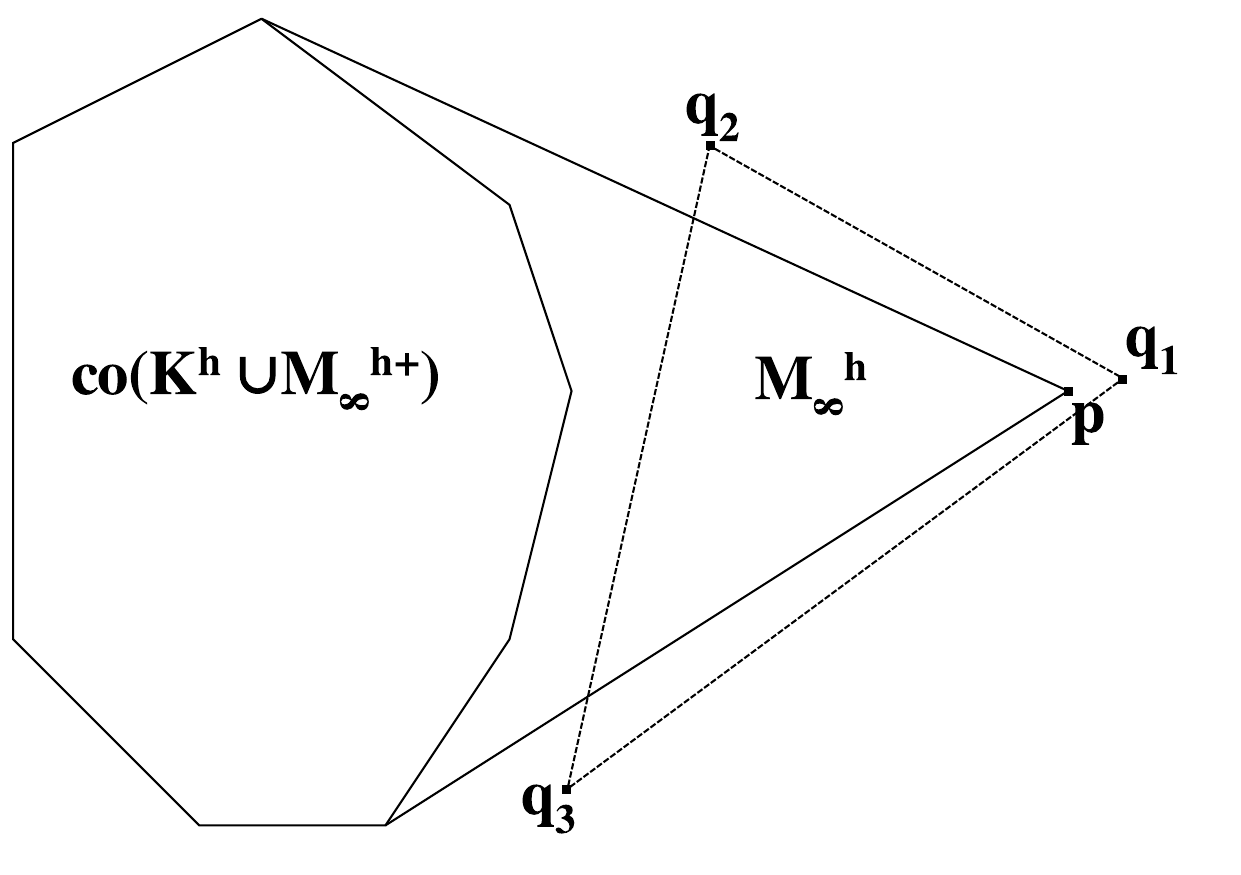}
      \caption{A $2+1$-prism for $(M,K)$ is used to remove a neighborhood of the point $p$}
      \label{fig: D prism that cuts p}
  \end{figure}

  Let $T$ be the triangle with vertices $q_1, q_2, q_3$.
  The product $P=T\times (h-\varepsilon, h+\varepsilon)$ is a  $2+1$-prism for $(M_\infty, K)$ because:
\begin{itemize}
  \item
  The closure of its vertical rectangles $\overline{q_1 q_2}\times(h-\varepsilon, h+\varepsilon)$ and $\overline{q_1 q_3}\times(h-\varepsilon, h+\varepsilon)$ do not intersect $ M_\infty^{h}\times(h-\varepsilon, h+\varepsilon)$, which contains both $ M_\infty^{h+}\times(h-\varepsilon, h+\varepsilon)$ and $M_\infty^{h-}\times(h-\varepsilon, h+\varepsilon)$.
  \item
  The closure of its upper triangle $T\times\{h-\varepsilon\}$ does not intersect $M_\infty^{h+}$.
  \item $P$ does not contain points of $K$.
\end{itemize}

  Since $M_\infty$ is the limit of the $M_n$, it follows that there is some $n_0$ such that $P$ is also a $2+1$-prism for $(M_n, K)$, for $n\geq n_0$.

  However, if $P$ is a $2+1$-prism for $(M_n, K)$, and since it contains $p$, then $p$ would be removed by the ($2+1$, $K$) reduction at height $h_j$ during step $n$.
  Indeed, if $\bar{p}$ is not removed by a ($2+1$, $K$) reduction at height $h_j$, that's because $\bar{p}\in \co \left(\left( M_n^{h_{j+1}}\cap M_n^{h_{j-1}}\right) \cup  K^{h_j}\right)$.
  We know that $p\not\in K$.
  If $\bar{p}\in  M_n^{h_{j+1}}\cap M_n^{h_{j-1}}$, then the vertical segment $\{\bar{p}\}\times(h_{j-1}, h_{j+1})\subset M_n$ would intersect both the upper and the lower triangle in $P$.
  Since $M_n$ is $2+1$ convex, $(M_n)^{h_{j+1}}\cap  (M_n)^{h_{j-1}}\subset (M_n)^{h_j}$, and $P$ would not be a $2+1$-prism for $(M_n,K)$.

  The only remaining possibility is that $p$ is in the interior of a segment with one endpoint $r_1\in K^{h_j}$ and the other $r_2\in M_n^{h_{j+1}}\cap  M_n^{h_{j-1}}\subset M_n^{h_j}$.
  The triangle $T$ cannot intersect the segment $\overline{r_1 r_2}$ at two points, since $P$ is a $2+1$-prism for $(M_n,K)$.
  Since $T$ cannot contain $r_1$ but must contain $p$, $T$ must intersect $\overline{r_1 p}$ at one point.
  It follows that $T$ must contain $r_2$, but then the segment $\{r_2\}\times(h_{j-1}, h_{j+1})\subset M_n$ intersects both the upper and the lower triangle in $P$.

  We have reached a contradiction that proves that all extremal points of $M_\infty$ belong to $K$.
  $\qed$

  \paragraph{\textbf{Claim:}} $M_\infty^{h}$ is a finite polygon for any $h\in \R$.

  Assume on the contrary that there is at least one $h$ such that $M_\infty^{h}$ is not a finite polygon, and let $h_0$ be the smallest such $h$.
  Let $E$ be the set of extremal points of $M\cap\{z=h_0\}$ for standard convexity.
  For each point $p=(x,y,h_0)\in E$, let $p^+$ be the point with the highest $z$ coordinate among all points in $M_\infty$ with the same projection as $p$ to the $\{z=0\}$ plane, and define $p^-$ analogously:
  $$
  h^{+} = \max\{h\in\R: (x,y,h)\in M_\infty\} \Rightarrow p^+=(x,y,h^+)
  $$
  $$
  h^{-} = \min\{h\in\R: (x,y,h)\in M_\infty\} \Rightarrow p^-=(x,y,h^-)
  $$
  Since $M_\infty$ is $2+1$-convex, the segment $[p^-,p^+]$ is contained in $M_\infty$.

  The set $E\setminus(\pi)^{-1}(F)$ consists of the points of $E$ that do not project to $F=\pi(K)$.
  For a point $p\in E\setminus(\pi)^{-1}(F)$, neither $p^+$ nor $p^-$ can be horizontally extremal, because then they would belong to $K$.
  The point $p^+$ belongs to the interior of a segment $[q_1^+, q_2^+]$
  contained in $(M_\infty)^{h^+} $.
  If $M_\infty\cap \{z=h^-\}$ were a 2D polygon and $p^-$ belonged to its interior, then a subsegment $(r_1,r_2)$ of $[q_1^+, q_2^+]$ would be contained in both $M_\infty^{h^+}$ and $M_\infty^{h^-}$, and hence $(r_1,r_2)\times[p^-,p^+]$ would be contained in $M_\infty$, and $p$ would not be an extremal point of $M_\infty\cap\{z=h_0\}$.

  Hence $p^-$ belongs to the relative boundary of $M_\infty\cap\{z=h^-\}$, for some $h^-<h_0$.
  The collection of all the sides of polygons $M_\infty\cap\{z=h\}$ for $h<h_0$ is finite, since there are only a finite number of different sets $M_\infty^{h}$, and by hypothesis, each $M_\infty^{h^-}$ is a finite polygon.
  Each side can contribute at most two extremal points of $M_\infty\cap\{z=h_0\}$ to $E$, since three aligned points cannot be all horizontally extremal.
  Hence $E\setminus(\pi)^{-1}(F)$ is finite.
  Since $F$ is finite and $E\subset\{z=h_0\}$, it follows that $E\cap(\pi)^{-1}(F)$ is also finite.
  Hence $E$ is finite, which is a contradiction, and it follows that $M_\infty\cap\{z=h\}$ is a finite polygon for any $h\in\R$.
  $\qed$

  \paragraph{\textbf{Conclusion:}}
  It follows that $M_\infty$ can be described as the union of a finite number of polygons at the heights in $H$ and a finite number of prisms of the form $P=C\times (h_j,h_{j+1})$ for a finite 2D polygon $C$.
  Thus $M_\infty$ is a $2+1$ $K$ complex.
\end{proof}

\begin{proof}[Proof of Theorem \ref{thm: maintheorem1}]
  By theorem \ref{thm: the kebab sequence consists of outer approximations}, $K^{rc}\subset K^{\pc}$.
  By theorem \ref{thm: the limit is a finite complex}, $K^{\pc}$ is a $2+1$ complex, which is by definition contained in the $2+1$ complexes convex hull $K^{cc}$ of $K$.
  By lemma \ref{lem:KcccontenidoKrc}, $K^{cc}\subset K^{rc}$.
  Hence $K^{\pc}\subset K^{cc}\subset K^{rc}\subset K^{\pc}$.
\end{proof}

\begin{proof}[Proof of Theorem \ref{thm: scaffolding}]
  The set of vertices of $K^{rc}\cap {z=h_j}$, for each $h_j\in H$ is a finite scaffolding for $K$.
\end{proof}

\begin{rem}
  The ideas in the proof of theorem \ref{thm: the limit is a finite complex} are enough to bound the number of points in the scaffolding of $K$ in terms only of the cardinal of $K$. However, without new ideas, the bounds would be exponential, and in any case they do not help find $K^{rc}$ in finite time.
\end{rem}

\section{Conclusions}\label{section:conclusions}

\subsection{On $2+1$-prisms}
We have chosen to use theorem \ref{thm:kirchheimballs} only with $2+1$-prisms, but that theorem can be used in other ways to improve an outer approximation.
If $K^{rc}\subset M$, then theorem \ref{thm:kirchheimballs} shows that
$$M'=(M\setminus B)\cup \left(\left[(B\cap K)\cup (\partial B\cap M)\right]^{rc}\cap B\right)$$
is also a (hopefully smaller) outer approximation.

A natural generalization of $2+1$-prisms is to consider sets $B$ such that $B\cap K=\emptyset$ and $\partial B\cap M$ is $2+1$ convex.
This generalization can provide improvements over $2+1$-prisms, but it is not obvious how to turn this use of theorem \ref{thm:kirchheimballs} into an algorithm.
It is not clear whether generalization to a family with finitely many parameters can find $K^{rc}$ exactly in finite time and memory.
Another approach that might work is to perform even more cuts with prisms in
each step.

On the one hand, the proof of theorem \ref{thm: the limit is a finite complex} limits the possible choices for points in each height, and hence the number of possible candidates for $K^{\pc}=K^{rc}$.
However, this is not practical, and we have some empirical evidence that it will not be easy to find the finite object $K^{rc}$ exactly in reasonable time and memory.
For any finite set $R\subset\R^2$, let $R'$ be the set of points of $\R^2$ that are the intersection of two segments with endpoints in $R$, and let $R^{(2)}=(R')'$ and $R^{(n+1)}=(R^{(n)})'$ for any natural $n$.
We have found examples of finite sets $K$ such that $K^{rc}$ can be computed through the outer algorithm in finitely many steps, and it contains vertices whose first two coordinates lie in $R^{(n)}$ for $n$ up to $4$.
Those examples are rather cumbersome, but while we try to understand them
better, they provide evidence that it will not be easy, and may not be possible at all, to find $K^{rc}$ exactly in polynomial time and with a polynomial memory footprint.

\subsection{Generalizations}
The definitions \ref{def: 2+1 complejo} of $2+1$-complex, and \ref{defn: 2+1 complexes convex hull} of $2+1$-complex convex hull, and also lemma \ref{lem:KcccontenidoKrc} are easy to generalize to arbitrary wave cones $D$.
Definition \ref{defn: local 2+1 prism outer approximation} can be extended to arbitrary wave cones $D$ in this way:
  The local $D$-ball convex hull $K^{D-\text{ball}}$ is the complement of the set of points that can be removed from $\co(K)$ by a finite sequence of applications of Theorem \ref{thm:kirchheimballs}.
  Another natural idea is to use only sets $B$ such that $B\cap K=\emptyset$ and $\partial B\cap M$ is $D$ convex.
  Although this definition is unsuitable for computation, it might work for theoretical applications.

  It does not seem difficult to prove that the $D$ extremal points of $K^{D-\text{ball}}$ would belong to $K$, but that knowledge would not be of much use.
  The following, and more interesting conjectures do not seem at all trivial, and are possibly not true in their full generality:
\begin{enumerate}
  \item For any wave cone $D$ there is a family of semialgebraic sets (called $D$-prisms), parameterized by finitely many real numbers, such that $K^{D-\text{ball}}$ is also the complement of the set of points that can be removed from $\co(K)$ by a finite sequence of applications of an adequate generalization of theorem \ref{thm:kirchheimprisms} (or a restriction of theorem \ref{thm:kirchheimballs}, as you wish to look at it).
  \item $K^{D-\text{ball}}=K^{rc}$.
  \item
  $K^{D-\text{ball}}$ is a finite $D$-complex (defined as a semi-algebraic set that is foliated by lines of $D$).
  \item For any wave cone $D$ and finite set $K$ there is a scaffolding $\widetilde{K}$ of finite order $n$: $K^{rc}=\widetilde{K}^{lc,n}$.

\end{enumerate}
The first of the above conjectures alone would give rise to an algorithm that could approximate $K^{rc}$ from outside.

\end{document}